\documentclass[11pt]{amsart}

\usepackage{amsmath,amssymb,amsthm,amsfonts}
\usepackage[all,cmtip]{xy}

\newcommand{\C}{\mathcal{C}}
\newcommand{\one}{\mathbf{1}}
\newcommand{\ot}{\otimes}

\DeclareMathOperator{\coh}{H}
\DeclareMathOperator{\Ext}{Ext}
\DeclareMathOperator{\HH}{HH}
\DeclareMathOperator{\Hom}{Hom}

\numberwithin{equation}{section}

\newtheorem{theorem}[equation]{Theorem}

\newtheorem{corollary}[equation]{Corollary}
\newtheorem{lemma}[equation]{Lemma}
\newtheorem{example}[equation]{Example}

\title[Hopf algebra cohomology]
{Lie brackets on Hopf algebra cohomology}

\author{Tek\.{i}n Karada\u{g}}
\address{Department of Mathematics, Texas A\&M University, 
College Station, Texas 77843, USA}
\email{tekinkaradag@math.tamu.edu}
\author{Sarah Witherspoon}
\address{Department of Mathematics, Texas A\&M University, 
College Station, Texas 77843, USA}
\email{sjw@math.tamu.edu}

\date{January 23, 2021}

\thanks{Partially supported by NSF grants 1665286 and 2001163.}


\begin{document}

\begin{abstract}
By work of Farinati, Solotar, and Taillefer,
it is known that the Hopf algebra cohomology of a quasi-triangular Hopf algebra,
as a graded Lie algebra under the Gerstenhaber bracket,
is abelian.
Motivated by the question of whether this holds for 
nonquasi-triangular Hopf algebras, 
we show that Gerstenhaber brackets on Hopf algebra cohomology
can be expressed via an arbitrary projective resolution using 
Volkov's homotopy liftings as 
generalized to some exact monoidal categories.
This is a special case of our more general result that
a bracket operation on cohomology is preserved under exact monoidal 
functors---one 
such functor is an embedding of Hopf algebra cohomology
into Hochschild cohomology. 
As a consequence, we show that this
Lie structure on Hopf algebra cohomology 
is abelian in positive degrees 
for all quantum elementary abelian
groups, most of which are nonquasi-triangular. 
\end{abstract}

\maketitle

\section{Introduction}\label{sec:introduction}


The Ext algebra of a Hopf algebra,
that is its Hopf algebra cohomology,
has a Lie structure arising from its embedding into Hochschild cohomology
in case the antipode is bijective.
This Lie structure is known to be abelian when the Hopf algebra
is quasi-triangular.
More specifically, Taillefer~\cite{Taillefer} constructed a Lie bracket
on cohomology arising from a category of Hopf bimodule extensions
of a Hopf algebra and showed that brackets are always zero.
In the finite dimensional case, this corresponds to Hopf algebra
cohomology of the opposite of the Drinfeld double, whose Lie structure
was also investigated by Farinati and Solotar~\cite{FS} using
different techniques. 
Farinati and Solotar  showed this is indeed the bracket arising
from an embedding into Hochschild cohomology. 
Hermann~\cite{Hermann} looked at a more general monoidal
category setting and showed that the Lie structure is
trivial in case the category is braided~\cite[Theorem 5.2.7]{Hermann}. 
It is an open question as to what happens more generally.

In this paper, we apply a new technique, the homotopy lifting method
of Volkov~\cite{Volkov}, to shed more light on this question,
particularly in the Hopf algebra setting.
Homotopy liftings were defined for 
some exact monoidal categories in~\cite{VW}, and we use them to prove
that brackets are preserved under exact monoidal functors. 
As a consequence, we show that the Lie bracket on Hopf algebra cohomology
defined by homotopy liftings agrees with 
that induced by its embedding into Hochschild cohomology,
in case the antipode is bijective. 
An advantage of the homotopy lifting method is its ability to handle
the Lie structure independently of choice of projective resolution.
A good choice of resolution 
can facilitate understanding of the Lie structure. 

We illustrate with the quantum elementary abelian groups. 
These are tensor products of copies of the Taft algebras $T_n$
that depend on an $n$th root of unity; when $n>2$ these are
nonquasi-triangular~\cite[Proposition~2.1]{Gunn}. 
The first author \cite{Karadag} found that the Lie bracket on
the Hopf algebra cohomology of $T_n$ 
is trivial by using techniques of Negron and the second author~\cite{NW},
a consequence of a larger computation of the bracket on
Hochschild cohomology of $T_n$.
This was the first known example of the Lie structure on
Hopf algebra cohomology of a nonquasi-triangular Hopf algebra.
In this paper we take a different approach, 
giving a second proof in positive degrees,
for Taft algebras $T_n$ via homotopy liftings
instead of relying on a computation for the full Hochschild cohomology ring. 
We extend this result to quantum elementary abelian groups,
as iterated tensor products of Taft algebras. 

An outline of the contents is as follows. 
In Section~\ref{sec:prelim},
we recall some definitions and the homotopy lifting method
for defining a Lie structure on the cohomology of some exact monoidal categories.
We use this method in Section~\ref{sec:change}, 
proving that the Lie structure is always 
preserved under an exact monoidal functor. 
In Section~\ref{sec:Hopf},
we apply the results of Section~\ref{sec:change} to Hopf algebra
cohomology, concluding a formula for Gerstenhaber brackets based
on homotopy liftings. 
The resulting Lie structure on the Hopf algebra cohomology
of a quantum elementary abelian group, in positive
degrees, is shown to be abelian 
in Section~\ref{sec:examples}. 

\section{Preliminaries}\label{sec:prelim}
We begin by recalling the definitions of Hopf algebra cohomology and Hochschild cohomology, as well as an embedding of one into the other
that provides a connection between their Lie structures.

Let $k$ be a field and abbreviate $\ot_k$ by $\ot$ when it will
cause no confusion. 

\subsection*{Hopf algebra cohomology}
 Let $A$ be a Hopf algebra over a field $k$ with 
coproduct $\Delta: A \rightarrow A\ot A$,
counit $\varepsilon:A\to k$ and antipode $S:A\to A$. 
{\em We will always assume that $S$ is bijective.}
Note that, since all finite dimensional and most known infinite dimensional Hopf algebras of interest have bijective antipodes, this assumption is not very restrictive.

Give the field $k$ the structure of an $A$-module, called the
{\em trivial} $A$-module, via the counit $\varepsilon$, 
i.e.~$a\cdot 1_k = \varepsilon(a)$ for all $a\in A$. 
Give $A$ the structure of an $A$-bimodule by multiplication on both sides;
equivalently, $A$ is a left $A^e$-module where $A^e := A\ot A^{op}$ and 
$A^{op}$ is the vector space $A$ with multiplication opposite to $A$. 
The {\em Hopf algebra cohomology} of $A$ is $\coh^*(A,k) := 
\Ext^*_A(k,k)$.
The {\em Hochschild cohomology} of $A$ is $\HH^*(A):=  \Ext^*_{A^e}(A,A)$.

We can embed Hopf algebra cohomology into Hochschild cohomology. 
We will need details, and we recall 
some lemmas. For  proofs, see \cite[Section 9.4]{HCSW}.
 
 \begin{lemma}\label{delta}
 There is an isomorphism of $A^e$-modules,
 $$A\cong A^e\ot_Ak,$$
 where $A^e\ot_Ak$ is the tensor induced $A^e$-module under the identification
of $A$ with the subalgebra of $A^e$ that is the image of the embedding 
$\delta:A\to A^e$ defined for all $a\in A$ by  
$$\delta(a)=\sum a_1\ot S(a_2).$$
 \end{lemma}
 
 \begin{lemma}\label{projective}
The right $A$-module $A^e$, where $A$ acts by right multiplication 
under its identification with $\delta(A)$, is  projective.
 \end{lemma}


\begin{lemma}[Eckmann-Shapiro]\label{ESL}
	Let $A$ be a ring and let $B$ be a subring of $A$ such that A is projective as a right $B$-module. Let $M$ be an $A$-module and $N$ be a $B$-module. Then $$\emph{Ext}_B^n(N,M)\cong \emph{Ext}_A^n(A\ot_BN,M), $$
where $M$ is considered to be a $B$-module under restriction.  
\end{lemma}
  
  We will consider $A$ to be a left $A$-module by the left adjoint action, which is for $a,b\in A$, 
  $$a\cdot b=\sum a_1bS(a_2) .$$
Denote this $A$-module by $A^{ad}$.

For any left $A$-module $M$, let $\coh^*(A,M):= \Ext^*_A(k,M)$. 
The following theorem is well known; see, e.g.~\cite[Theorem~9.4.5]{HCSW}. 
   We sketch a proof since we will need some of the details later. 
   
  \begin{theorem}\label{iso}
  There is an isomorphism of graded  $k$-vector spaces
  $$\HH^*(A)\cong \coh^*(A,A^{ad}).$$
  \end{theorem}

\begin{proof}
By Lemma \ref{projective}, $A^e$ is a projective as a right $A$-module, so we can apply the Eckmann-Shapiro Lemma. We replace $A$ with $A^e$, $B$ with $A$ and take $M=A$, $N=k$ in the Eckmann-Shapiro Lemma and obtain the isomorphism $\Ext_{A^e}^n(A^e\ot_Ak,A)\cong\Ext_A^n(k,A^{ad})$ as $k$-vector spaces. Lastly, we apply Lemma \ref{delta} and obtain $\Ext_{A^e}^n(A,A)\cong\Ext_A^n(k,A^{ad})$.
\end{proof}

A consequence of the theorem is an embedding of Hopf algebra cohomology
$\coh^*(A,k)$ into Hochschild cohomology $\HH^*(A)$:
Let $P\rightarrow k$ be a projective resolution of the $A$-module $k$. 
We can embed H$^*(A,k)$ into $\coh^*(A,A^{ad})\cong \HH^*(A)$ 
via the map $\eta_*:\text{Hom}_A(P,k)\to \text{Hom}_A(P,A^{ad})$ induced by the unit map $\eta: k\rightarrow A$
(see~\cite[Corollary 9.4.7]{HCSW}).
Equivalently, viewing each space $\coh^n(A,k)$ as equivalence classes 
of $n$-extensions of $k$ by $k$, the functor $A^e\ot_A -$ induces an
embedding of $\coh^*(A,k)$ into $\HH^*(A)$. 

\subsection*{Homotopy liftings and exact monoidal categories}
Gerstenhaber defined the bracket on Hochschild cohomology 
via the bar resolution. 
We will not need that definition here.
Instead, 
by using homotopy liftings as defined below, a similar operation is constructed for some exact monoidal categories in \cite[Section 4]{VW} as follows. This turns out to be equivalent to the Gerstenhaber bracket when the category is that of $A$-bimodules. 

We refer to~\cite{EGNO} for definitions and properties of monoidal categories.
Let $\C$ be an exact monoidal category and 
let $\one$ be its unit object.
As is customary, we will identify $\one\ot X$ and $X\ot \one$ with $X$
for all objects $X$ in $\C$, under assumed fixed isomorphisms (for which we will
not need notation).
One example of an exact monoidal category is the category of left
modules for a Hopf algebra over $k$, with tensor product $  \ot  $
of modules, and the unit object given by the trivial module $k$.
Another example is the category of $A$-bimodules (equivalently left
$A^e$-modules) for an associative algebra $A$ over $k$,
with tensor product $  \ot_A $ and unit object $A$.

Let $P\rightarrow \one$ be a projective resolution of $\one$ with
differential $d$ and let 
$\mu_P:P_0\to \one$ be the corresponding augmentation map. 
In~\cite[Definition~4.3]{VW}, 
the resolution $(P,d,\mu_P)$ of $\one$ is called $n$-\textit{power flat} 
if $(P^{\ot r},d^{\ot r},\mu_P^{\ot r})$ is a projective resolution of $\one$ for each $r$ ($1\leq r\leq n$). 
If $P$ is $n$-power flat for each $n\geq2$, then we say that $P$ is \textit{power flat}.
For the two categories mentioned above, projective resolutions of $\one$ are generally power flat. 

Assume $\one$ has a projective power flat resolution $P$.
For a degree $l$ morphism $\psi:P\rightarrow P$,
its differential in the Hom complex $\Hom_{\C}(P,P)$
is defined to be
\[
  \partial ( \psi) = d\psi - (-1)^l \psi d . 
\]
Let $f:P\to \one$ be an $m$-cocycle. 
Let $\Delta_P: P\rightarrow P\ot P$ be a diagonal map, i.e.~a chain map lifting the isomorphism $\one\stackrel{\sim}{\longrightarrow} \one\ot\one$. 
A degree ($m-1$) morphism $\psi_f:P\to P$ is a \textit{homotopy lifting of} ($f,\Delta_P$) if
\begin{equation}\label{eqn:psi-f}
\partial(\psi_f)=(f\ot1_P-1_P\ot f)\Delta_P
\end{equation}
and $\mu_P\psi_f\sim (-1)^{m+1}f\psi$ for some degree $-1$ map $\psi:P\to P$ such that
\begin{equation}\label{eqn:psi}
  \partial(\psi)=(\mu_P\ot1_P-1_P\ot \mu_P)\Delta_P  .
\end{equation}

The cohomology of the monoidal category $\C$ is 
$\coh^*(\C) = \coh^*(\C, \one) := \Ext^*_{\C}(\one , \one)$. 
It has a Lie bracket defined as follows~\cite[Section 4]{VW}: 
For an $m$-cocycle $f:P_m\to \one$ and an $n$-cocycle $g:P_n\to \one$, let $\psi_f$ and $\psi_g$ be homotopy liftings of ($f,\Delta_P$) and ($g,\Delta_P$) respectively. Then the cochain $[f,g]$ defined as 
\begin{equation}\label{eqn:bracket} 
[f,g]=f\psi_g-(-1)^{(m-1)(n-1)}g\psi_f
\end{equation}
induces a graded Lie bracket on $\coh^*(\C)$.
That is, it induces 
a well-defined operation on cohomology that is graded alternating and satisfies a graded Jacobi identity (cf.~\cite[Lemma 1.4.3]{HCSW}).

\section{Change of exact monoidal categories}\label{sec:change}

Let $\C$, $\C'$ be exact monoidal categories 
for which there exist power flat resolutions of their unit
objects $\one$, $\one '$. 
Let $F: \C \rightarrow \C'$ be an exact monoidal functor~\cite[Definition 2.4.1]{EGNO},
that is, $F$ is exact and there is a natural isomorphism $\eta$ of functors from $\C\times \C$ to $\C'$
given by 
\[
\eta_{X,Y}: F(X)\ot F(Y)\rightarrow  F(X\ot Y)  
\]
for all $X$, $Y$ in $\C$,  $F(\one)\cong \one'$ and $(F,\eta)$ satisfies the monoidal structure axiom of~\cite[Definition~2.4.1]{EGNO}, that is, the following diagram commutes for all objects $X,Y,Z$ in $\C$:
\begin{equation}\label{eqn:monoidal}
\xymatrix{
(F(X)\ot F(Y))\ot F(Z)\ar[d]^{\eta_{X,Y}\ot 1_{F(Z)}} \ar[rr]^{\sim} 
    && F(X)\ot (F(Y)\ot F(Z))\ar[d]^{1_{F(X)}\ot\eta_{Y,Z}}\\
F(X\ot Y)\ot F(Z)\ar[d]^{\eta_{X\ot Y,Z}} && F(X)\ot F(Y\ot Z)\ar[d]^{\eta_{X,Y\ot Z}}\\
F((X\ot Y)\ot Z) \ar[rr]^{\sim}&& F(X\ot (Y\ot Z))
}
\end{equation}
(The horizontal isomorphisms are given by the associativity constraint
for $\C'$ and the image of the associativity constraint for $\C$ under $F$,
respectively. We will not need notation for these isomorphisms.)

Denote the isomorphism from $F(\one)$ to $\one '$ by $\phi$. Then the following diagrams commute for
all objects $X$ by~\cite[Proposition 2.4.3]{EGNO}; we have chosen to show
diagrams involving 
the inverse maps $\eta^{-1}_{\one, X}$ and $\eta^{-1}_{X, \one}$ since
we will need these later.
The unlabeled isomorphisms in the diagrams are those
canonically determined by the fixed isomorphisms given by 
tensoring with unit objects and the fixed
isomorphism $F(\one) \cong \one '$. 
\begin{equation*}
\xymatrixcolsep{3pc}
\xymatrix{
F(\one\ot X)\ar[r]^{\quad \sim}\ar[d]_{\eta^{-1}_{\one,X}} & F(X)\ar[d]^{\sim}
   &&   F(X\ot \one)\ar[r]^{\quad \sim}\ar[d]_{\eta^{-1}_{X,\one}} & F(X)\ar[d]^{\sim}\\
F(\one)\ot F(X) \ar[r]^{\hspace{.4cm}\sim} & \one ' \ot F(X) && F(X)\ot F(\one)\ar[r]^{\sim} & F(X)\ot \one '
}
\end{equation*}

\quad

\begin{example}{\em 
Let $A$ be a Hopf algebra and $B$ a Hopf subalgebra of $A$.
Let $\C$ be the category of left $A$-modules and $\C'$
the category of left $B$-modules.
Let $F: \C\rightarrow \C'$ be the restriction functor,
that is on each $A$-module $X$, the action is restricted to $B$ 
via the inclusion map $B\hookrightarrow A$.
The restriction of a tensor product of modules to $B$
is isomorphic to the tensor product of their restrictions to $B$,
and thus there is a 
natural transformation $\eta$ as required.}
\end{example}

In the next section we will apply the following theorem
to shed light on the connection between the Lie structures on Hopf algebra
cohomology and on Hochschild cohomology.

Let $P$ be a projective resolution of $\one$ in $\C$ and write
$P'=F(P)$, which is a projective resolution of $F(\one)\cong\one'$ 
in $\C'$ under our assumptions. 
Let $d$ denote the differential and $\mu_P:P\rightarrow \one$ denote the augmentation map of $P$. Write $d'=F(d)$ and $\mu_{P'}=F(\mu_P)$. 
Note that $P\ot P$ is also a projective resolution of $\one$
in $\C$ with augmentation map $\mu_P\ot\mu_P$ followed
by the canonical isomorphism $\one\ot\one\stackrel{\sim}{\longrightarrow}\one$. 
Let
\[
   \Delta_{P'} = \eta^{-1}_{P,P} F(\Delta_P) ,
\]
which is a diagonal map on $P'$ under our assumptions.

\begin{theorem}\label{thm:main}
Let $\C$, $\C'$ be exact monoidal categories and let $F: \C\rightarrow \C'$
be an exact monoidal functor. 
Assume there exists a power flat resolution $P$ of $\one$ in $\C$.
Let $f\in\Hom_{\C}(P_m, \one)$, an $m$-cocycle. 
Let $\psi_f$ be a homotopy lifting of $f$ with respect to $\Delta_P$.
Then  $F(\psi_f)$ is a homotopy lifting of $F(f)$ with respect to $\Delta_{P'}$. 
\end{theorem}

\begin{proof}
Since $\psi_f$ is a homotopy lifting of $f$, 
\[
     d \psi_f  - (-1)^{m-1} \psi_{f}d =  (f\ot 1 - 1\ot f) \Delta_P .
\]
Set $f'=F(f)$, $\psi_{f'} = F(\psi_f)$, and apply $F$ to each side of this
equation to obtain
\begin{equation}\label{right}
   d'\psi_{f'} - (-1)^{m-1} \psi_{f'} d' = F(f\ot 1 - 1\ot f) F(\Delta_P) .
\end{equation}
Since $\eta$ is a natural transformation, under our assumptions (see
the above commuting diagrams), 
the following diagram commutes: 

\begin{equation*}\label{def:tau}
\xymatrix{F(P)\ar[r]^-{F(\Delta_P)}&F(P\ot P) \ar[r]^-{F(f\ot 1)} \ar[d]_-{\eta^{-1}_{P,P}} & F(\one \ot P)
     \ar[r]^{\hspace{.2cm}\sim}\ar[d]_-{\eta^{-1}_{\one,P}} & F(P)\ar[d]^{=} \\
	 & F(P)\ot F(P) \ar[r]^-{F(f)\ot 1} & F(\one)\ot F(P)\ar[r]^{\hspace{.5cm}\sim}& F(P) }
\end{equation*}
Therefore $F(f\ot 1) F(\Delta_P)$ can be identified with
$(F(f)\ot 1)\eta^{-1}_{P,P} F(\Delta_P)$, and similarly
$F(1\ot f)F(\Delta_P)$ with $(1\ot F(f))\eta^{-1}_{P,P} F(\Delta_P)$. 
So the right side of expression (\ref{right}) is equal to 
\[
   (f'\ot 1 - 1\ot f' ) \eta^{-1}_{P,P} F(\Delta_P) ,
\]
which is in turn equal to 
$ (f'\ot 1 - 1\ot f')  \Delta_{P'}$, 
as desired. 

Since $\psi_f$ is a homotopy lifting, $\mu_P\psi_f\sim (-1)^{m+1}f\psi$ for some degree $-1$ map $\psi:P\to P$ such that $$\partial(\psi)=d\psi+\psi d=(\mu_P\ot1_P-1_P\ot \mu_P)\Delta_P.$$
By applying $F$ to both sides of the above equation, we obtain
\[
   F(d\psi+\psi d)=F(\mu_P\ot1_P-1_P\ot \mu_P)F(\Delta_P) , 
\]
and under our identifications, letting $\psi'= F(\psi)$, this is
\[
d'\psi'+\psi' d' =F(\mu_P\ot 1)F(\Delta_P)-F(1\ot \mu_P)F(\Delta_P) .
\]
Via a commutative diagram such as that above, we see this is equal to
\[ 
  (\mu_{P'}\ot 1)\Delta_{P'}-(1\ot \mu_{P'})\Delta_{P'} .
\]
Therefore, $\psi':P'\to P'$ is a degree $-1$ map such that 
$$\partial(\psi')=d'\psi'+\psi'd'=(\mu_{P'}\ot 1)\Delta_{P'}-(1\ot \mu_{P'})\Delta_{P'} , $$
and $\mu_{P'}F(\psi_f)\sim (-1)^{m+1}F(f)\psi'$,
that is, $\mu_{P'} \psi_{f'} \sim  (-1)^{m+1} f' \psi '$, as desired. 

We have shown that 
$F(\psi_f)$ is a homotopy lifting of $F(f)$ with respect to $\Delta_{P'}$.
\end{proof}

\begin{corollary}
The functor $F$ induces a graded Lie algebra homomorphism
from $\coh^*(\C)$ to $\coh^*(\C')$, in positive degrees.
\end{corollary}

\begin{proof}
As a consequence of the theorem and formula~\eqref{eqn:bracket},
the functor $F$ takes the Lie bracket of two elements of positive degree
in the cohomology
$\coh^*(\C)$ to the Lie bracket of their images
in $\coh^*(\C')$ under $F$.
\end{proof}

\section{Hopf algebra cohomology}\label{sec:Hopf}
Let $A$ be a Hopf algebra with bijective antipode. 
Let $\C$ be the category of (left) $A$-modules, and let $\C'$ be the category
of (left) $A^e$-modules. 
For each $A$-module $U$, let
\[
    F(U) = A^e \ot_A U ,
\]
the tensor induced module, 
where we identify $A$ with the subalgebra $\delta(A)$ of $A^e$ as in Lemma~\ref{delta}.
Also by Lemma~\ref{delta}, $F$ takes the unit object $k$ of $\C$ to
an isomorphic copy of the unit object $A$ of $\C'$.
It takes projective $A$-modules to projective $A^e$-modules
since $A^e$ is projective as an $A$-module by Lemma~\ref{projective}.
For each $A$-module homomorphism $f: U\rightarrow V$, define $F(f)$ by
\[
     F(f) ((1\ot 1)\ot_A u) =  (1\ot 1) \ot_A f(u)
\]
for all $u\in U$. Then $F$ may be viewed as 
the functor providing the embedding of Hopf algebra
cohomology $\coh^*(A,k)$ into Hochschild cohomology $\HH^*(A)$;
see the proof of Theorem~\ref{iso} and the subsequent paragraph.

For each pair of $A$-modules $U$, $V$, we wish to define
an $A^e$-module homomorphism 
\[
   \eta_{U,V}: F(U)\ot_A F(V)\rightarrow F(U\ot V)  , 
\]
that is, 
\[
   \eta_{U,V}:   
(A^e\ot_A U)\ot_A (A^e\ot_A V) 
\rightarrow 
A^e \ot_A (U\ot V) .
\]
For all $a,b\in A$, $u\in U$, and $v\in V$, set 
\[
   \eta_{U,V} 
((a\ot 1)\ot_A u)\ot_A ((1\ot b)\ot_A v)=
(a\ot b)\ot_A (u\ot v) .
\]
Note that all elements of $(A^e\ot_A U)\ot_A (A^e\ot_A V)$ can indeed
be written as linear combinations of elements of the indicated forms
and that the map is well-defined. 
For example, for all $a,b\in A$ and $u\in U$, letting $b = S(b')$,
\begin{eqnarray*}
   (a\ot b) \ot_A u &=& \sum (a S(b_1') b_2'\ot S(b_3'))\ot_A u\\
      &=& \sum (a S(b_1')\ot 1)\ot _A ((b_2'\ot S(b_3'))\cdot u) .
\end{eqnarray*}
By its definition, $\eta_{U,V}$ is an $A^e$-module homomorphism.

We check that $\eta$ is a natural transformation.
That is, the following diagram commutes for all objects
$U, V, U', V'$ and morphisms $f: U\rightarrow U'$, $g: V\rightarrow V'$:
\begin{equation*}
\xymatrix{
F(U)\ot_A F(V)
\ar[d]_{\eta_{U,V}}
\ar[rr]^{F(f)\ot_AF(g)} 
&& F(U')\ot_A F(V')\ar[d]^{\eta_{U',V'}}\\
F(U\ot V)\ar[rr]^{F(f\ot g)} 
&& F(U'\ot V')
}
\end{equation*}
Commutativity follows from the definitions of $\eta_{U,V}$, $\eta_{U',V'}$. 
To see that $\eta$ is monoidal, that is diagram (\ref{eqn:monoidal}) commutes,
it is easier to check the corresponding diagram associated to $\eta^{-1}$,
a straightforward calculation. 

For the following theorem,
we define the {\em  Gerstenhaber bracket} of two elements in
Hopf algebra cohomology $\coh^*(A,k)$ via the embedding into
Hochschild cohomology followed by the Gerstenhaber bracket on Hochschild cohomology.
The theorem states that this is the same
as their bracket defined by~(\ref{eqn:bracket}) 
on Hopf algebra cohomology $\coh^*(A,k)$ 
via homotopy liftings.
Thus the 
theorem allows us to 
bypass the need to work with Hochschild cohomology at all,
for questions purely about Hopf algebra cohomology. 

\begin{theorem}\label{thm:bracket-formula}
Let $A$ be a Hopf algebra with bijective antipode.
Let $P$ be a projective resolution of $k$ and let 
$f$, $g$ be cocycles  in $\Hom_A(P_m,k)$, $\Hom_A(P_n,k)$, respectively,
representing elements of Hopf algebra cohomology $\coh^*(A,k)$.
Let $\Delta_P$ be a diagonal map,
and let $\psi_f$, $\psi_g$ be homotopy liftings of $f,g$ with respect to $\Delta_P$.
The Gerstenhaber bracket of the corresponding elements in Hopf algebra cohomology
$\coh^*(A,k)$ is represented by
\[
   [f,g] = f\psi_g - (-1)^{(m-1)(n-1)} g\psi_f .
\]
\end{theorem}

\begin{proof}
This is an immediate consequence of Theorem~\ref{thm:main}
and expression~(\ref{eqn:bracket}),
since we showed above that the induction functor $F$ is
an exact monoidal functor.
\end{proof}

One consequence Theorem~\ref{thm:bracket-formula} is a quick new proof that
for a cocommutative Hopf algebra in characteristic not 2, 
Gerstenhaber brackets on Hopf algebra cohomology in positive degree are always 0.
We state this as Corollary~\ref{cor:abelian} next.
Since cocommutative Hopf algebras are quasi-triangular, this is a small
special case of the well known results of Farinati, Solotar, Taillefer, and
Hermann, but it highlights our completely different approach. 

\begin{corollary}\label{cor:abelian}
Let $k$ be a field of characteristic not 2,
and let $A$ be a cocommutative Hopf algebra.
The Lie structure on Hopf algebra cohomology $\coh^*(A,k)$,
given by the Gerstenhaber bracket, is abelian in positive degrees.
\end{corollary}

\begin{proof}
Let $P$ be a projective resolution of $k$ as an $A$-module.
Let $\Delta': P\rightarrow P\ot P$ be a diagonal map.
Let $\sigma: P\ot P\rightarrow P\ot P$ be the signed transposition map,
i.e.~$\sigma(x\ot y) = (-1)^{|x||y|} y\ot x$ for all homogeneous $x,y\in P$.
Since $A$ is cocommutative, $\sigma \Delta'$ is also an $A$-module
homomorphism, and therefore a diagonal map.
Let
\[
    \Delta  = \frac{1}{2} (\Delta' + \sigma \Delta') ,
\]
a diagonal map as well. Note that $\Delta $ is symmetric in the sense
that $\sigma\Delta = \Delta$. 

Now, by symmetry, $(\mu_P\ot 1_P - 1_P\ot \mu_P)\Delta \equiv 0$,
and so in~(\ref{eqn:psi}), 
we can take $\psi \equiv 0$. 
Similarly, in (\ref{eqn:psi-f}), we can take 
$\psi_f \equiv 0$ for
any cocycle $f$.
Thus by Theorem~\ref{thm:bracket-formula}, Gerstenhaber brackets on the Hopf
algebra cohomology $\coh^*(A,k)$ are always 0 in positive degrees.
\end{proof}

\section{Taft algebras and
quantum elementary abelian groups}\label{sec:examples}

In this section we illustrate our results by showing that 
the Lie structure on the Hopf algebra cohomology of a Taft algebra,
and more generally of a quantum elementary abelian group, 
is abelian in positive degrees. 

Let $k$ be a field of characteristic 0
containing a primitive $n$th root $\omega$ of 1. 
Let $A$ be the Taft algebra generated by $x$ and $g$ with relations
\[
     g x = \omega x g , \ \ \ x^n=0 , \ \ \ g^n=1 .
\]
We take
\[
   \Delta (x) = x\ot 1 + g\ot x , \ \ \  \Delta(g) = g\ot g ,
\]
$\varepsilon(x)=0$, $\varepsilon(g)=1$, $S(x)= -g^{-1}x$, $S(g) = g^{-1}$.
Let $B = k[x]/(x^n)$, a subalgebra of $A$, and let 
$P$ be the following projective resolution of the 
trivial $B$-module $k$:
\[
\xymatrix{
\cdots\ar[r]^{x^{n-1}\cdot } & B \ar[r]^{x\cdot} & 
B\ar[r]^{x^{n-1}\cdot} & B\ar[r]^{x\cdot} & B\ar[r]^{\varepsilon} & k
\ar[r] & 0 
}
\]
The action of $B$ on each term is by multiplication. 
Give each component $B$ in the resolution
the structure of an $A$-module by letting 
$g\cdot x^i = \omega^i x^i$ in even degrees and
$g\cdot x^i = \omega^{i+1} x^i$ in odd degrees.
For clarity of notation, in each degree $l$,
denote the element $1_B$  of $P_l = B$ by $\epsilon_l$.
Note that $P_l$ is projective as an $A$-module since
the characteristic of $k$ is not divisible by $n$:
Specifically,
in even degrees, there are $A$-module homomorphisms $P_l\rightarrow A$
($x^i \mapsto  \frac{1}{n}\sum_{j=0}^{n-1} x^i g^j$) and
$A\rightarrow P_l$ ($x^i g^j \mapsto x^i$) whose composition
is the identity map.
In odd degrees, a similar statement is true of the maps $P_l\rightarrow A$
and $A\rightarrow P_l$ given respectively by
\[
   x^i \mapsto \left\{ \begin{array}{ll}\frac{1}{n}\sum_{j=0}^{n-1}
      x^{i+1}g^j, & \mbox{ if } i<n-1, \\
    \frac{1}{n}\sum_{j=0}^{n-1}  g^j, & \mbox{ if } i=n-1 
     \end{array}\right. 
    \ \ \ \mbox{ and } \ \ \
    x^i g^j \mapsto \left\{\begin{array}{ll}
    x^{i-1}, & \mbox{ if } i\neq 0 , \\ 
      x^{n-1}, & \mbox{ if } i =0. 
   \end{array}\right. 
\]

Calculations show that the following formulas yield
a diagonal map $\Delta: P\rightarrow P\ot P$,
that is for each $l$, $\Delta_l$ is an $A$-module homomorphism,
and $\Delta$ is a chain map lifting the canonical isomorphism 
$k\stackrel{\sim}{\longrightarrow}k\ot k$: 
\begin{eqnarray*}
\Delta_{2j+1} (\epsilon_{2j+1}) & = & 
   \sum_{i=0}^{2j+1}   \epsilon_i\ot \epsilon_{2j+1-i},\\
\Delta_{2j} (\epsilon_{2j} ) &=& \sum_{i=0}^j  \epsilon_{2i}
\ot \epsilon_{2j-2i} 
       + \sum_{i=0}^{j-1}\sum_{a=0}^{n-2}  
  \binom{n-1}{a+1}_\omega x^a\epsilon_{2i+1} \ot x^{n-2-a}\epsilon_{2j-2i-1},
\end{eqnarray*}
where $\binom{n-1}{a+1}_\omega$ is the $\omega$-binomial coefficient
defined for all nonnegative integers $a,b,c$ by
\[
    \binom{b}{c}_\omega =\frac{ (b)_\omega (b-1)_\omega\cdots
   (b-c+1)_\omega}{ (c)_\omega (c-1)_\omega\cdots (1)_\omega} 
   \ \ \ \mbox{ where } \ \ \
    (a)_\omega = 1 + \omega + \omega^2 + \cdots + \omega^{a-1} .
\]

Note that $\Delta\neq \sigma \Delta$ since $\binom{n-1}{a+1}_\omega
\neq \binom{n-1}{a}_\omega$ in general.
However, symmetry does hold
after projection onto even degrees, a key property for the proof
of the theorem below since the cohomology is concentrated
in even degrees as we see next. 

The cohomology of $A$ can be computed directly from the resolution
$P$ above, and is
\[
   \coh^*(A,k) \cong \coh^*(B, k)^G \cong k[z] ,
\]
where $\deg(z)=2$.
Alternatively, see~\cite[Corollary 3.4]{Stefan} for the relevant general theory
for skew group algebras. 

The following theorem was proven in~\cite{Karadag} by different techniques,
and more generally there, the elements of degree~0 were included.
The homotopy lifting method that we use here was designed 
for positive degree cohomology. 

\begin{theorem}
The Lie structure given by the Gerstenhaber bracket 
on the cohomology $\coh^*(A,k)$ of a Taft
algebra $A$ is abelian in positive degrees.
\end{theorem}

\begin{proof}
Let $P$ be the resolution of $k$ given above. 
Let $f\in\Hom_A (P_2,k)$ denote the cocycle with 
$f(\epsilon_2) =1$, a representative of the generator $z$
of the cohomology ring $\coh^*(A,k)$, described above. 
By~\cite[Lemma 1.4.7]{HCSW}, it will suffice to show that
the bracket of $f$ with itself is~0 since
$f$ represents an algebra generator of cohomology. 

We wish to find a homotopy lifting of $f$. 
First note that in (\ref{eqn:psi}), we can take $\psi \equiv 0$,
the zero map, by symmetry of the image of the diagonal map
under the projection onto $(P_0\ot P_i) \oplus (P_i\ot P_0)$ for each $i$. 
Similarly, by symmetry of the image of the diagonal map
under the projection onto $(P_{\rm{even}}\ot P)
\oplus (P\ot P_{\rm{even}})$, since $f$ has even degree, in (\ref{eqn:psi-f}),
we can take $\psi_f \equiv 0$
and indeed we can take the homotopy lifting of any 
representative element of cohomology in positive degree to be 0. 
Specifically, the map $\psi_f$ must satisfy
\[
   d\psi_f + \psi_f d =  (f\ot 1-1\ot f)\Delta_P .
\] 
The right side of this equation, evaluated on 
$\epsilon_l$, is
\[(f\ot 1 - 1\ot f) (\Delta_P(\epsilon_l)),\]
and comparing to the formulas for $\Delta_P(\epsilon_{2j})$
and $\Delta_P(\epsilon_{2j+1})$ above, 
the only terms that will be nonzero after applying $f\ot 1 
-1\ot f$ are those having $\epsilon_2$ as one of the tensor factors.
By symmetry, the resulting terms after applying
$f\ot 1$ and $1\ot f$ cancel due to their opposite signs. 
So we may take $\psi_f\equiv 0$ as claimed. 
Now, by Theorem~\ref{thm:bracket-formula},
${[} f , f {]} = 2 f \psi_f = 0$. 
\end{proof}

The following theorem is a consequence since the Lie structure of a 
tensor product of algebras reduces to that on each factor~\cite{LZ}.
Quantum elementary abelian groups are defined to
be iterated tensor products of Taft algebras~\cite{PW}.

\begin{theorem}
Let $A$ be a quantum elementary abelian group.
The Lie structure of the Hopf algebra cohomology $\coh^*(A,k)$, given
by Gerstenhaber bracket, 
is abelian in positive degrees.
\end{theorem}

\end{document}